\documentclass[reqno,12pt]{amsart}
\usepackage{amscd}
\usepackage{amssymb}
\usepackage{amsmath}
\usepackage{amsfonts}
\usepackage{enumerate}
\usepackage{graphicx}
\usepackage{epic}
\usepackage[all,cmtip]{xy}
\usepackage{overpic}
\usepackage{subfig}
\usepackage{caption}
\usepackage{tikz}

\newtheorem{thm}{Theorem}[section]
\newtheorem{prop}[thm]{Proposition}
\newtheorem{lem}[thm]{Lemma}
\newtheorem{cor}[thm]{Corollary}
\newtheorem{conj}[thm]{Conjecture}
\theoremstyle{definition}
\newtheorem{Def}[thm]{Definition}

\newtheorem*{rem*}{Remark}

\newcommand{\kk}{\mathbf{k}}
\newcommand{\xx}{\mathbf{x}}

\newcommand{\Aa}{\mathcal {A}}
\newcommand{\BB}{\mathcal {B}}

\newcommand{\CC}{\mathcal {C}}

\def\w{\widetilde}

\numberwithin{equation}{section}
\usepackage[colorlinks,linktocpage]{hyperref}
\hypersetup{citecolor=blue,linkcolor=blue}
\begin{document}
	\title[The generic anisotropy of s.e.d. spheres]{The generic anisotropy of strongly edge decomposable spheres}
	\author[F.~Fan]{Feifei Fan}
	\thanks{The author is supported by the National Natural Science Foundation of China (Grant no. 12271183) and by the GuangDong Basic and Applied Basic Research Foundation (Grant no. 2023A1515012217).}
	\address{Feifei Fan, School of Mathematical Sciences, South China Normal University, Guangzhou, 510631, China.}
	\email{fanfeifei@mail.nankai.edu.cn}
	\maketitle
	\begin{abstract}
The generic anisotropy is an important property in the study of Stanley-Reisner rings of homology spheres, which was introduced by Papadakis and Petrotou. 
This property can be used to prove the strong Lefschetz property as well as McMullen's $g$-conjecture for homology spheres. It is conjectured that for an arbitrary field $\mathbb{F}$, any $\mathbb{F}$-homology sphere is generically anisotropic over $\mathbb{F}$.
In this paper, we prove this conjecture for all strongly edge decomposable spheres.
\end{abstract}

	\section{Introduction}\label{sec:introduction}
	The conception of generic anisotropy was introduced by Papadakis and Petrotou \cite{PP20}. As shown in \cite{PP20}, the generic anisotropy implies the strong Lefschetz property (see Subsection \ref{subsec:lefschetz}), as well as the Hall-Laman
	relations defined by Adiprasito \cite{A18}, of the general Artinian reductions of the
	Stanley–Reisner rings of simplicial homology spheres in the sense of \cite[Definition 2.2]{PP20}. 
	
	Assume $\mathbb{F}$ is a field, 
	and $\Delta$ is a $\mathbb{F}$-homology sphere of dimension $d-1$ with vertex set $\{1,\dots,m\}$.
Denote by $\kk=\mathbb{F}(a_{i,j})$ the field of rational functions in the variables $a_{i,j}$, $1\leq i\leq d$, $1\leq j\leq m$. Let $\kk[\Delta]$ be the Stanley–Reisner ring of $\Delta$ over $\kk$, a quotient ring of the polynomial ring $\kk[x_1,\dots,x_m]$ (see subsection \ref{subsec:l.s.o.p.}), and let $\kk(\Delta)=\kk[\Delta]/(\theta_1,\dots,\theta_d)$, where $\theta_i$ is the linear form $\sum_{j=1}^m a_{i,j}x_{j}$. 
Then $\Delta$ is said to be \emph{generically anisotropic over $\mathbb{F}$}, if
for all integers $j$ with $1\leq 2j\leq d$ and all nonzero elements $u\in\kk(\Delta)_j$ we have $u^2\neq 0$.

It is known that strong Lefschetz property implies the celebrated $g$-conjecture characterizing the face numbers of homology spheres (see \cite{KN16}), hence
the following theorem of Papadakis and Petrotou provides a second proof of $g$-conjecture for $\mathbb{F}$-homology spheres, $\mathbb{F}$ a field of characteristic $2$. An earlier proof of the $g$-conjecture for more general rational homology spheres was given by  Adiprasito \cite{A18}.
	\begin{thm}[{Papadakis and Petrotou, \cite{PP20}}]\label{thm:PP}
		Let $\mathbb{F}$ be a field of characteristic $2$. Then all $\mathbb{F}$-homology spheres are generically anisotropic over $\mathbb{F}$.
	\end{thm}
	To extend this theorem to arbitrary characteristics, there is the following natural conjecture suggested by Adiprasito, Papadakis and Petrotou \cite{APP21}.
	\begin{conj}[Anisotropy conjecture]
		For an arbitrary field $\mathbb{F}$, any $\mathbb{F}$-homology sphere is generically anisotropic over $\mathbb{F}$.
	\end{conj}
	So far, this conjecture is widely open and only
	the special case of simplicial spheres of dimension $1$ was proved by Papadakis and Petrotou \cite{PP23}. 
	In the present paper, we will show
	that this conjecture is true for strongly edge decomposable spheres, which are a class of PL-spheres introduced by Nevo \cite{Nev07}.
	
 Let $\Delta$ be a  simplicial complex on $[m]=\{1,\dots,m\}$, a collection of subsets of $[m]$ (including $\emptyset$) that is closed under inclusion. The elements $\sigma\in\Delta$ are called \emph{faces} and
 the maximal faces of $\Delta$ under inclusion are called \emph{facets}.
 The \emph{dimension} of a face $\sigma\in\Delta$ is
 $\dim\sigma=|\sigma|-1$ ($\dim\emptyset=-1$) and the dimension of $\Delta$ is $\dim\Delta= \max\{\dim\sigma: \sigma\in\Delta\}$.
   The \emph{link} and \emph{star} of a face $\sigma\in\Delta$ are respectively the subcomplexes
	\[\begin{split}
		\mathrm{lk}_\Delta\sigma=&\{\tau\in\Delta:\tau\cup\sigma\in\Delta,\tau\cap\sigma=\emptyset\};\\
		\mathrm{st}_\Delta\sigma=&\{\tau\in\Delta:\tau\cup\sigma\in\Delta\}.
	\end{split}\]
	If $\sigma=\{i,j\}$ is an edge ($1$-dimensional face) of $\Delta$, the \emph{contraction} of $\Delta$ with respect to $\sigma$ is the simplicial complex
	$\CC(\sigma,\Delta)$ on $[m]\setminus\{i\}$ which is obtained from $\Delta$ by identifying the vertices $i$ and $j$. More precisely,
	\[\CC(\sigma,\Delta)=\{\tau\in \Delta :i\not\in\tau\}\cup\{(\tau\setminus\{i\})\cup\{j\}:i\in\tau\in\Delta\}.\]
	We say that $\Delta$ satisfies the \emph{link condition with respect to $\sigma$} if
	\[\mathrm{lk}_\Delta\{i\}\cap \mathrm{lk}_\Delta\{j\}=\mathrm{lk}_\Delta\sigma.\]
	
	\begin{Def}
		A simplcial sphere $\Delta$ is said to be \emph{strongly edge decomposable}  (s.e.d. for short) if
		either $\Delta$ is the boundary of a simplex or else there exists an edge
		$\sigma\in\Delta$ such that $\Delta$ satisfies the link 
		condition with respect to $\sigma$ and both $\mathrm{lk}_\Delta\sigma$ and $\CC(\sigma,\Delta)$ are s.e.d.
	\end{Def}
	
Murai \cite{Murai10}  proved that s.e.d. spheres have the strong Lefschetz property over any infinite field. (The strong Lefschetz property of s.e.d. spheres in characteristic zero was also proved by Babson and Nevo \cite{BN10}.) The following theorem, the main result of this paper, shows that a stronger property holds for s.e.d. spheres.
\begin{thm}\label{thm:sed}
	All	s.e.d. spheres are generically anisotropic over any field $\mathbb{F}$. 
\end{thm}

This paper is organized as follows. In Section \ref{sec:preliminaries}, we introduce the basic notions and collect some known results about Stanley–Reisner rings. In Section \ref{sec:equivalence}, we establish an equivalent but simpler condition for a homology sphere to be generically anisotropic (Theorem \ref{thm:aniso equiv}) and provide an application (Proposition \ref{prop:sus}), which are key ingredients in the proof of Theorem \ref{thm:sed}. The main idea in the proof of Theorem \ref{thm:sed} is to show that the generic anisotropy of an odd dimensional homology sphere satisfying the link condition is preserved by the contraction if its contraction and its link satisfy some nice algebraic properties (Proposition \ref{prop:sed}). The even dimensional case can be reduced to the odd dimensional case by a combinatorial result (Proposition \ref{prop:sed sus}). These results and the proof of  Theorem \ref{thm:sed} are given in Section \ref{sec:sed}.

	\section{Preliminaries}\label{sec:preliminaries}

	\subsection{The Stanley-Reisner ring}\label{subsec:l.s.o.p.}
Let $\Delta$ be a  simplicial complex on $[m]$.	For a field $\kk$, let $S=\kk[x_1,\dots,x_m]$ be the polynomial algebra with one generator for each
	vertex in $\Delta$. It is a graded algebra by setting $\deg x_i=1$.
	The \emph{Stanley-Reisner ideal} of $\Delta$ is
	\[I_\Delta:=(x_{i_1}x_{i_2}\cdots x_{i_k}:\{i_1,i_2,\dots,i_k\}\not\in\Delta)
	\]
	The \emph{Stanley-Reisner ring} (or \emph{face ring}) of $\Delta$ is the quotient \[\kk[\Delta]:=S/I_\Delta.\]
	Since $I_\Delta$ is a monomial ideal, the quotient ring $\kk[\Delta]$ is graded by degree.
	
	For a face $\sigma=\{i_1,\dots,i_k\}\in\Delta$, denote by $\xx_\sigma=x_{i_1}\cdots x_{i_k}\in\kk[\Delta]$ the face monomial corresponding to $\sigma$. Sometimes we also use $\sigma$ to mean $\xx_{\sigma}$ for simplicity.
	
A sequence $\Theta=(\theta_1,\dots,\theta_d)$ of $d=\dim\Delta+1$ linear forms in $S$ is called a \emph{linear
	system of parameters} (or \emph{l.s.o.p.} for short) if
	\[\kk(\Delta;\Theta):=\kk[\Delta]/(\Theta)\]
	has Krull dimension zero, i.e., it is a finite-dimensional $\kk$-space.
	The quotient ring $\kk(\Delta;\Theta)$ is an \emph{Artinian reduction} of $\kk[\Delta]$. We will use the simplified notation $\kk(\Delta)$ for $\kk(\Delta;\Theta)$ whenever it creates no confusion, and write the component
	of degree $i$ of $\kk(\Delta)$ as $\kk(\Delta)_i$. For a subcomplex $\Delta'\subset\Delta$, let $I$ be the ideal of $\kk[\Delta]$ generated by faces in $\Delta\setminus\Delta'$, and denote $I/I\Theta$ by $\kk(\Delta,\Delta';\Theta)$ or simply $\kk(\Delta,\Delta')$.
	
	A linear sequence $\Theta=(\theta_1,\dots,\theta_d)$ is an l.s.o.p if and only if the restriction $\Theta_\sigma=r_\sigma(\Theta)$ to each face $\sigma\in\Delta$ generates the polynomial algebra $\kk[x_i:i\in\sigma]$; here $r_\sigma:\kk[\Delta]\to\kk[x_i:i\in\sigma]$ is the projection homomorphism. If $\Theta$ is an l.s.o.p. for $\kk[\Delta]$, then $\kk(\Delta)$ is spanned by the face monomials (see \cite[Theorem 5.1.16]{BH98}). If $\kk$ is an infinite field, then $\kk[\Delta]$ admits an l.s.o.p. (Noether normalization lemma). 
	
	A simplicial complex $\Delta$ is called \emph{Cohen-Macaulay over $\kk$} if for any l.s.o.p $\Theta=(\theta_1,\dots,\theta_d)$, $\kk[\Delta]$ is a free $\kk[\theta_1,\cdots,\theta_d]$ module. By a result of Reisner \cite{Rei76}, $\Delta$ is Cohen-Macaulay over $\kk$ if and only if for all faces $\sigma\in\Delta$ (including $\sigma=\emptyset$)
	and $i<\dim\mathrm{lk}_\Delta\sigma$, we have $\w H_i(\mathrm{lk}_\Delta\sigma;\kk)=0$. It follows from \cite[II. Corollary 2.5]{S96} that if $\Delta$ is a Cohen-Macaulay (over $\kk$) complex of dimenison $d-1$, then for any Artinian reduction $\kk(\Delta)$, $\dim_\kk\kk(\Delta)_i=h_i(\Delta)$, $0\leq i\leq d$,
	where $h_i(\Delta)$ is a combinatorial invariant of $\Delta$ defined by
	\[h_i(\Delta)=\sum_{j=0}^{i}(-1)^{i-j}\binom{d-j}{d-i}f_{j-1},\]
	and $f_i$ is the number of $i$-dimensional faces of $\Delta$.
	
	Suppose that $\Theta=(\theta_i=\sum_{j=1}^m a_{i,j}x_j)_{i=1}^d$ is an l.s.o.p. for $\kk[\Delta]$. Then there is an associated $d\times m$ matrix $M_\Theta=(a_{i,j})$.
Let $\boldsymbol{\lambda}_i=(a_{1,i},a_{2,i},\dots,a_{d,i})^T$ denote the column vector corresponding to the vertex $i\in[m]$. For any ordered subset $I=(i_1,\dots,i_k)\subset [m]$, the submatrix $M_\Theta(I)$ of $M_\Theta$ is defined to be
\[M_\Theta(I)=(\boldsymbol{\lambda}_{i_1},\dots,\boldsymbol{\lambda}_{i_k}).\]
	
	\subsection{Strong Lefschetz and Anisotropy}\label{subsec:lefschetz}
	The Lefschetz property of face rings is strongly connected to many topics in algebraic geometry, commutative algebra and combinatorics. For instance, the strong Lefschetz property for homology spheres is the algebraic version of the $g$-conjecture in the most strong sence, a generalization of the Hard Lefschetz Theorem in algebraic geometry for projective toric variaties. 
	
	A $(d-1)$-dimensional simplicial complex
	$\Delta$ is a \emph{$\kk$-homology sphere}, $\kk$ a field, if for every face $\sigma\in\Delta$ (including $\sigma=\emptyset$) the link of $\sigma$ has
	the same reduced homology as a sphere of dimension $d-2-\dim\sigma$:
	\[\w H_i(\mathrm{lk}_\Delta\sigma;\kk)=\begin{cases}
		\kk&\text{if } i=d-2-\dim\sigma,\\
		0&\text{otherwise.}
	\end{cases}.\]
 It is known that if $\Delta$ is a $\kk$-homology $(d-1)$-sphere, then the face ring $\kk[\Delta]$ is a \emph{Gorenstein ring} (see \cite[II. Theorem 5.1]{S96}), i.e., $\kk(\Delta;\Theta)$ is a Poincar\'e duality $\kk$-algebra for any l.s.o.p $\Theta$ in the sence that the multiplication map $\kk(\Delta;\Theta)_i\times \kk(\Delta;\Theta)_{d-i}\to \kk(\Delta;\Theta)_d\cong\kk$
 is a perfect pairing for $0\leq i\leq d$. (See \cite[p. 50]{S96} for other equivalent definitions.)
	
	\begin{Def}\label{def:SL}
		Let $\Delta$ be a $\kk$-homology $(d-1)$-sphere. We say that $\kk[\Delta]$ has the \emph{strong Lefschetz property}  (or $\Delta$ is strong Lefschetz over $\kk$) 
		if there is an Artinian reduction $\kk(\Delta;\Theta)$ of $\kk[\Delta]$ and a linear form
		$\omega\in\kk[\Delta]_1$ such that  the multiplication map
	\[\cdot\omega^{d-2i}:\kk(\Delta;\Theta)_{i}\to\kk(\Delta;\Theta)_{d-i}\]
 is bijective for $i=0,1,\dots,\lfloor\frac{d}{2}\rfloor$. The element $\omega$ is called a \emph{strong Lefschetz element} of $\kk(\Delta;\Theta)$.
	\end{Def}

Note that the set of $(\Theta,\omega)$ in Definition \ref{def:SL} is
Zariski open (see e.g. \cite[Propostion 3.6]{Swa06}), but it may be empty. 
If $\Delta$ is strong Lefschetz over $\kk$, then it is also strong  Lefschetz over any infinite field with the same characteristic as $\kk$  (see e.g. \cite[Propostion 13.6]{PP20}).

As we mentioned in Section \ref{sec:introduction}, in order to prove the strong Lefschetz property of homology spheres in characteristic $2$, Papadakis and Petrotou \cite{PP20} established the anisotropy of their face rings over a large field.

\begin{Def}\label{def:anisotropy}
	Let  $\Delta$ be a $\kk$-homology $(d-1)$-sphere. An Artinian reduction $\kk(\Delta)$ of $\kk[\Delta]$ is said to be \emph{anisotropic} if for every nonzero element $u\in\kk(\Delta)_i$ with $i\leq d/2$, the square $u^2$ is also nonzero in $\kk(\Delta)_{2i}$. We call $\Delta$ \emph{anisotropic over $\kk$} if such an Artinian reduction exists.
\end{Def}

It turns out that anisotropy is stronger than the strong Lefschetz property in the sence that if a class of $\kk$-homology spheres, which is closed under the suspension operation, is anisotropic over $\kk$, then any $\kk$-homology sphere in this class is strong Lefschetz (see \cite[Section 9]{PP20}). 
Here the suspension $S(\Delta)$ of a simplicial complex $\Delta$ means the join $S^0*\Delta$, where $S^0$ is the simplicial complex consisting of two disjoint vertices. Recall that the join of two simplicial complexes $\Delta$ and $\Delta'$ with disjoint sets of vertices is
$\Delta*\Delta':=\{\sigma\cup\tau : \sigma\in\Delta,\,\tau\in\Delta'\}$. 
	
\subsection{Canonical modules for homology balls} In this subsection we recall some results about the \emph{canonical module} (see \cite[I.12]{S96} for the definition) of $\kk[\Delta]$ when $\Delta$ is a $\kk$-homology ball. 

First we recall the definition of $\kk$-homology manifold. A $d-1$ dimensional simplicial complex $\Delta$ is a \emph{$\kk$-homology manifold (without boundary)} if
for any nonempty faces $\sigma\in\Delta$, $\mathrm{lk}_\Delta\sigma$ is a $\kk$-homology sphere of dimension $d-2-\dim\sigma$.
Similarly, $\Delta$ is a \emph{$\kk$-homology manifold with boundary} if the link of every nonempty face $\sigma$ has the homology of a $(d-2-\dim\sigma)$-dimensional sphere or a ball (over $\kk$), and the
boundary complex of $\Delta$, i.e.,
\[\partial\Delta:=\{\sigma\in\Delta:\w H_*(\mathrm{lk}_\Delta\sigma;\kk)=0\}\cup\{\emptyset\}\]
is a $(d-2)$-dimensional $\kk$-homology manifold without boundary. 
A connected $\kk$-homology manifold $\Delta$ of dimension $d-1$ is \emph{orientable} if $\w H_{d-1}(\Delta,\partial\Delta;\kk)\cong\kk$. In this case, an orientation of $\Delta$ is given by an ordering on the vertices of each facets of $\Delta$.

A $\kk$-homology $(d-1)$-ball is a $(d-1)$-dimensional $\kk$-homology manifold with boundary whose homology is trivial and whose
boundary complex is a $\kk$-homology $(d-2)$-sphere.

Let $\Delta$ be a $\kk$-homology $(d-1)$-ball with boundary $\partial \Delta$. Then there is an exact sequence
\begin{equation}\label{eq:exact}
	0\to I\to \kk[\Delta]\to\kk[\partial\Delta]\to 0,
\end{equation}
where $I$ is the ideal of $\kk[\Delta]$ generated by all faces in $\Delta\setminus\partial\Delta$. By a theorem of Hochster (see \cite[II. Theorem  7.3]{S96}) $I$ is the canonical module of $\kk[\Delta]$. Then from \cite[Theorem 3.3.4 (d) and Theorem 3.3.5 (a)]{BH98} we have the following
\begin{prop}\label{prop:pairing}
	Let $\Delta$ be a $\kk$-homology $(d-1)$-ball. Then for any Artinian reduction, the multiplication map
	\[
	\kk(\Delta)_i\times \kk(\Delta,\partial\Delta)_{d-i}\to \kk(\Delta,\partial\Delta)_d\cong\kk
	\]
is a perfect pairing for $0\leq i\leq d$.
\end{prop}
This result has the following useful corollary.

	\begin{cor}\label{cor:exact}
	Let $K$ be a $\kk$-homology sphere and $\Delta\subset K$ be a $\kk$-homology ball of the same dimension. Then for any Artinian reduction $\kk(K)$ and its restriction to $\Delta$, there is a short exact sequence 
\[0\to\kk(K,\Delta)\to\kk(K)\to\kk(\Delta)\to0.\]
\end{cor}
\begin{proof}
	Cf. the first paragraph of the proof of \cite[Theorem 3.1]{Swa14}.
\end{proof}

\subsection{Lee's formula}\label{subsec:Lee's} 
In this subsection, we introduce a formula due to Lee that expresses non-square-free monomials in $\kk(\Delta)$ in terms of face monomials. First, we recall a useful result in \cite{PP20}.

\begin{lem}[{\cite[Corollary 4.5]{PP20}}]\label{lem:generator}
	Let $\Delta$ be a $(d-1)$-dimensional $\kk$-homology sphere or ball, $\Theta$ be an l.s.o.p. for $\kk[\Delta]$. Suppose that $\sigma_1$ and $\sigma_2$ are two ordered facets of $\Delta$, which have the same orientation in $\Delta$. Then \[\det(M_\Theta(\sigma_1))\cdot\xx_{\sigma_1}=\det(M_\Theta(\sigma_2))\cdot\xx_{\sigma_2}\] in $\kk(\Delta)_d$ or in $\kk(\Delta,\partial\Delta)_d$.
\end{lem}

When $\Delta$ is a $(d-1)$-dimensional $\kk$-homology sphere or ball, $\kk(\Delta)_d$ or $\kk(\Delta,\partial\Delta)_d$ is  $\kk$, which is spanned by a facet monomial. So each facet $\sigma\in\Delta$ defines a map \[\Psi_\sigma:\kk(\Delta)_d\ \text{ or }\  \kk(\Delta,\partial\Delta)_d\to \kk\]
such that for all $\alpha$ in $\kk(\Delta)_d$ or $\kk(\Delta,\partial\Delta)_d$,
\[\alpha=\Psi_\sigma(\alpha)\det(M_\Theta(\sigma))\xx_{\sigma}.\]
Lemma \ref{lem:generator} says that $\Psi_\sigma=\pm\Psi_\tau$ for  any two facets $\sigma,\tau\in\Delta$.
If we fix an orientation on $\Delta$, this map is independent of the choice of the ordered facet whose ordering is compatible with the given orientation on $\Delta$, defining a \emph{canonical function} $\Psi_\Delta:\kk(\Delta)_d\text{ or } \kk(\Delta,\partial\Delta)_d\to \kk$ (see \cite[Remark 4.6]{PP20}). In particular, if $\sigma$ is a facet of $\Delta$, then $\Psi_\Delta(\xx_\sigma)=1/\det(M_\Theta(\sigma))$.

To state Lee's formula, we will need the following notation. Under the assumption of Lemma \ref{lem:generator}, let $\mathbf{a}=(a_1,\dots,a_d)^T\in\kk^d$ be a vector such that every $d\times d$ minor of the  matrix $(M_\Theta\mid\mathbf{a})$ is nonsingular. For any ordered subset $I\subset[m]$ with $|I|=d$, let $A_I=\det(M_\Theta(I))$, and for any $i\in I$, denote by $A_I(i)$ the determinant of the matrix obtained from $M_\Theta(I)$ by replacing the column vector $\boldsymbol{\lambda}_i$ with $\mathbf{a}$.

\begin{thm}[Lee's formula {\cite[Theorem 11]{Lee96}}]\label{thm:Lee}
	Let $\Delta$ be a $(d-1)$-dimensional $\kk$-homology sphere (resp. $\kk$-homology  ball), and fix an orientation on $\Delta$. Then for a monomial $x_{i_1}^{r_1}\cdots x_{i_k}^{r_k}\in\kk(\Delta)_d$ (resp.  $x_{i_1}^{r_1}\cdots x_{i_k}^{r_k}\in\kk(\Delta,\partial \Delta)_d$), $r_i>0$, we have
	\[\Psi_\Delta(x_{i_1}^{r_1}\cdots x_{i_k}^{r_k})=\sum_{\text{facets } F\in\mathrm{st}_\Delta\sigma}\frac{\prod_{i\in\sigma}A_F(i)^{r_i-1}}{A_F\prod_{i\in F\setminus\sigma}A_F(i)},\]
	where $\sigma=\{i_1,\dots,i_k\}$ and the sum is over all ordered facets of $\mathrm{st}_\Delta\sigma$, whose orderings are compatible with the given orientation on $\Delta$.
\end{thm}

\section{Equivalent anisotropy conditions}\label{sec:equivalence}
As shown in Theorem \ref{thm:PP}, if $\mathbb{F}$ is a field of characteristic $2$ and $\Delta$ is a $\mathbb{F}$-homology $(d-1)$-sphere with $m$ vertices, then $\kk(\Delta;\Theta)$ is anisotropic for the field of rational functions
\[\kk:=\mathbb{F}(a_{i,j}:1\leq i\leq d,\,1\leq j\leq m)\]
and the l.s.o.p. $\Theta=(\theta_i=\sum_{i=1}^m a_{i,j}x_j)_{i=1}^d$.
In fact, under the assumption of Theorem \ref{thm:PP} $\Delta$ is anisotropic over some smaller field extension of $\mathbb{F}$ than $\kk$, as we will see below.

Let 
\[\kk'=\mathbb{F}(a_{i,j}:1\leq i\leq d,\,d+1\leq j\leq m),\]
and denote by $A$ the $d\times (m-d)$ matrix $(a_{i,j})$. One easily sees that there is an l.s.o.p. $\Theta'$ for $\kk'[\Delta]$ such that $M_{\Theta'}=( I_d\mid A)$, where $I_d$ is the $d\times d$ identity matrix.

\begin{thm}\label{thm:aniso equiv}
	Suppose that $\Delta$ is a $\mathbb{F}$-homology ($\mathrm{char}\,\mathbb{F}$ is arbitrary) $(d-1)$-sphere with $m$ vertices. Let $\kk$, $\Theta$ and $\kk'$, $\Theta'$ be as above. 
	Then $\kk(\Delta;\Theta)$ is anisotropic if and only if $\kk'(\Delta;\Theta')$ is anisotropic.
\end{thm}
\begin{proof}
	``$\Rightarrow$".	There exists a matrix $N\in GL(d,\kk)$ such that $N M_\Theta=(I_d\mid B)$, where $B=(b_{i,j})$ ($1\leq i\leq d,\,d+1\leq j\leq m$) is a $d\times (m-d)$ matrix  with entries $b_{i,j}\in\kk$. Denote by $\Theta_0$ the l.s.o.p. corresponding to $(I_d\mid B)$. Clearly, the two ideals generated by $\Theta$ and $\Theta_0$ are the same. Let 
	\[\kk_0=\mathbb{F}(b_{i,j}:1\leq i\leq d,\,d+1\leq j\leq m).\]
	Then $\kk_0$ is a subfield of $\kk$. One easily sees that $b_{i,j}$ are algebraically independent elements over $\mathbb{F}$, so there is an isomorphism $\kk'\cong\kk_0$ given by $a_{i,j}\mapsto b_{i,j}$, and then an induced isomorphism $\kk'(\Delta;\Theta')\cong \kk_0(\Delta;\Theta_0)$. Since $\kk_0(\Delta;\Theta_0)\subset \kk(\Delta;\Theta_0)=\kk(\Delta;\Theta)$ 
	and $\kk(\Delta;\Theta)$ is anisotropic, any nonzero element 
	$u\in\kk_0(\Delta;\Theta_0)_i\cong \kk'(\Delta;\Theta')_i$ with $i\leq d/2$ satisfies $u^2\neq0$. Hence $\kk'(\Delta;\Theta')$ is anisotropic.
	
	``$\Leftarrow$". Pick an arbitrary order on the variables $a_{i,j}$ for $1\leq i\leq d$, $1\leq j\leq d$, and rewrite them as $a_1,a_2,\dots,a_{d^2}$. Let $\kk_0=\kk'$, and recursively define $\kk_i=\kk_{i-1}(a_i)$, i.e. the field of fractions of $\kk_{i-1}[a_i]$,  for $1\leq i\leq d^2$. Hence there is a sequence of field extension \[\kk'=\kk_0\subset\kk_1\subset\cdots\subset\kk_{d^2}=\kk.\]
	 Let $\Theta_0=\Theta'$ be the l.s.o.p. for $\kk_0[\Delta]$. For $1\leq i\leq d^2$, if $a_{i}=a_{j,k}$, then define $\Theta_i$ to be the l.s.o.p. for $\kk_i[\Delta]$ such that $M_{\Theta_i}$ is obtained from $M_{\Theta_{i-1}}$ by replacing the $(j,k)$-entry by $a_{j,k}$. 
	
	We will prove that $\kk_{i}(\Delta;\Theta_i)$ are all anisotropic for $0\leq i\leq d^2$ by induction on $i$. The base case $i=0$ is just the assumption. For the induction step, set $R_i=\kk_0[a_1,\dots,a_i]$, and denote by $\mathfrak{p}_i\subset R_i$ the prime ideal 
	\[\mathfrak{p}_i=\begin{cases}
		(a_i-1)&\text{if  $a_i=a_{k,k}$ for some $1\leq k\leq d$},\\
		(a_i)&\text{otherwise.}
	\end{cases}\] 
	Then there is a ring homomorphism $\eta_i:(R_i)_{\mathfrak{p}_i}\to \kk_{i-1}$, where $(R_i)_{\mathfrak{p}_i}\subset\kk_i$ denote the localization of $R_i$ at $\mathfrak{p}_i$, given by $\eta_i(a_j)=a_j$ for $1\leq j\leq i-1$, and 
	\[\eta_i(a_i)=\begin{cases}
		1&\text{if  $a_i=a_{k,k}$ for some $1\leq k\leq d$},\\
		0&\text{otherwise.}
	\end{cases}\]
	Clearly, $\eta_i$ induces a homomorphism $(R_i)_{\mathfrak{p}_i}(\Delta;\Theta_i)\to \kk_{i-1}(\Delta;\Theta_{i-1})$, which we also denoted by $\eta_i$.
	
	Given a nonzero element $\alpha\in\kk_{i}(\Delta;\Theta_i)_j$ with $j\leq d/2$, we claim that there exists a nonzero element $t\in\kk_i$ such that
	\[t\alpha\in(R_i)_{\mathfrak{p}_i}(\Delta;\Theta_i)\ \text{ and }\ 0\neq\eta_i(t\alpha)\in\kk_{i-1}(\Delta;\Theta_{i-1}).\]
	Assume this claim for the moment. Since $\kk_{i-1}(\Delta;\Theta_{i-1})$ is anisotropic by induction, $[\eta_i(t\alpha)]^2\neq0$. It follows that $t^2\alpha^2$ is not zero in $(R_i)_{\mathfrak{p}_i}(\Delta;\Theta_i)$, and then $0\neq\alpha^2\in\kk_{i}(\Delta;\Theta_i)$. So $\kk_{i}(\Delta;\Theta_i)$ is anisotropic, and the induction step is completed.
	
	It remains to prove the claim. Recall that $\kk_{i-1}(\Delta;\Theta_{i-1})$ is spanned by face monomials. Suppose that $\{\xx_{\sigma_1},\dots,\xx_{\sigma_s}\}$ is a basis of $\kk_{i-1}(\Delta;\Theta_{i-1})_j$. Then it is also a basis of $\kk_i(\Delta;\Theta_i)_j$. To see this, let $f=\sum_{r=1}^s l_r\xx_{\sigma_r}$ be a nontrivial $\kk_i$-linear combination of the $\xx_{\sigma_r}$. Then there is a power $t$ of $a_i$ (or $a_i-1$) such that $tl_r\in (R_i)_{\mathfrak{p}_i}$ for all $r$ and $\eta_i(tl_q)\neq 0$ for some $q$. Hence $tf\in (R_i)_{\mathfrak{p}_i}(\Delta;\Theta_i)$ and $\eta_i(tf)\neq 0$ in $\kk_{i-1}(\Delta;\Theta_{i-1})$. This implies that $f\neq 0$ in $\kk_i(\Delta;\Theta_i)$. Since $\Delta$ is a Cohen-Macaulay complex, $\dim_{\kk_i}\kk_i(\Delta;\Theta_i)_j=\dim_{\kk_{i-1}}\kk_{i-1}(\Delta;\Theta_{i-1})_j=h_j(\Delta)$. Therefore $\xx_{\sigma_1},\dots,\xx_{\sigma_s}$ forms a basis of $\kk_i(\Delta;\Theta_i)_j$. Hence, for the element $\alpha$ in the previous paragraph, we may assume $\alpha=f$, and take $t$ to be as above, then the claim is verified.
\end{proof}

Here is an application of Theorem \ref{thm:aniso equiv}.

\begin{prop}\label{prop:sus}
	Let $\Delta$ be a $\mathbb{F}$-homology $(2n-2)$-sphere on $[m]$. 
	If the suspension $S(\Delta)$ of $\Delta$ is generically anisotropic over $\mathbb{F}$, then $\Delta$ is also generically anisotropic over $\mathbb{F}$.
\end{prop}
\begin{proof}
	Write $S(\Delta)=K\cup K'$, where $K=\{m+1\}*\Delta$ and $K'=\{m+2\}*\Delta$. Let $\kk$ be the rational function field 
	\[\mathbb{F}(a_{i,j}:1\leq i\leq 2n,\,1\leq j\leq m),\]
	and let $\kk_0\subset\kk$ be the subfield 
	\[\mathbb{F}(a_{i,j}:2\leq i\leq 2n,\,1\leq j\leq m).\]
	Choose an l.s.o.p. $\Theta=(\theta_1,\theta_2,\dots,\theta_{2n})$ for $\kk[S(\Delta)]$ such that $M_\Theta=(A\mid\boldsymbol{\lambda}_{m+1},\boldsymbol{\lambda}_{m+2})$, where $A=(a_{i,j})$, and $(\boldsymbol{\lambda}_{m+1},\boldsymbol{\lambda}_{m+2})=(I_2\mid 0)^T$.
	Let $\Theta_0=(\theta_2,\theta_3,\dots,\theta_{2n})$. Clearly, $\Theta_0$ restricted to $\Delta$ is an l.s.o.p. for $\kk[\Delta]$.
	It is known that there are two isomorphisms:
	\begin{gather*}
		\kk(\Delta;\Theta_0)\cong\kk(K;\Theta),\ x_i\mapsto x_i\text{ for }1\leq i\leq m,\\
		\kk(\Delta;\Theta_0)_*\cong\kk(K,\Delta;\Theta)_{*+1},\ \alpha\mapsto x_{m+1}\alpha
	\end{gather*}
	(see e.g. \cite[Lemma 3.2 and 3.3]{A18}). Hence for a nonzero element $\alpha\in\kk(\Delta;\Theta_0)$, we have $0\neq x_{m+1}\alpha\in\kk(K,\Delta;\Theta)$. 
	
	Assume  $S(\Delta)$ is generically anisotropic over $\mathbb{F}$. Then $\kk(S(\Delta);\Theta)$ is anisotropic by the proof of Theorem \ref{thm:aniso equiv}. For any nonzero element $\alpha\in\kk_0(\Delta;\Theta_0)_i\subset \kk(\Delta;\Theta_0)_i$, $i\leq n-1$, the second isomorphism above shows that $0\neq x_{m+1}\alpha\in\kk(K,\Delta;\Theta)_{i+1}$.  Hence we have  
	$0\neq(x_{m+1}\alpha)^2\in\kk(K,\Delta;\Theta)$, since $\kk(K,\Delta;\Theta)=\kk(S(\Delta),K';\Theta)\subset\kk(S(\Delta);\Theta)$ by Corollary \ref{cor:exact} and $\kk(S(\Delta);\Theta)$ is anisotropic. 
	This means that $x_{m+1}\alpha^2$ is not zero in $\kk(K;\Theta)$ since by Proposition \ref{prop:pairing} there exists $\beta\in \kk(K;\Theta)$ such that $(x_{m+1}\alpha)^2\beta=x_{m+1}\alpha^2\cdot x_{m+1}\beta\neq 0$ in $\kk(K,\Delta;\Theta)_{2n}$, and we can think of $x_{m+1}\alpha^2\in \kk(K;\Theta)$ and $x_{m+1}\beta\in \kk(K,\Delta;\Theta)$ in the pairing in Proposition \ref{prop:pairing}.
	Then $0\neq\alpha^2\in\kk_0(\Delta;\Theta_0)\subset \kk(\Delta;\Theta_0)$ because of the first isomorphism above $\kk(K;\Theta)\cong\kk(\Delta;\Theta_0)$. So $\kk_0(\Delta;\Theta_0)$ is anisotropic. This is equivalent to saying that $\Delta$ is generically anisotropic over $\mathbb{F}$ by definition.
\end{proof}

\section{Proof of Theorem \ref{thm:sed}}\label{sec:sed}
In this section $\mathbb{F}$ denotes a field of arbitrary characteristic. 

Theorem \ref{thm:sed} is a consequence of the following two propositions.

\begin{prop}\label{prop:sed sus}
	If $\Delta$ is an s.e.d. $(d-1)$-sphere, then the suspension $S(\Delta)$ is an s.e.d. $d$-sphere. 
\end{prop}

\begin{prop}\label{prop:sed}
	Let $\Delta$ be a $\mathbb{F}$-homology $(2n-1)$-sphere on $[m]$ satisfying the link condition with respect to an edge $\sigma\in\Delta$. If $\mathrm{lk}_\Delta\sigma$ and $\CC(\sigma,\Delta)$ are generically anisotropic over $\mathbb{F}$ and $\mathrm{lk}_\Delta\sigma$ is strong Lefschetz over an infinite field $\kk$ ($\mathrm{char}\,\kk=\mathrm{char}\,\mathbb{F}$), then $\Delta$ is generically anisotropic over $\mathbb{F}$.
\end{prop}

Assuming Propositions \ref{prop:sed sus} and \ref{prop:sed}, we  prove Theorem \ref{thm:sed} as follows. First we consider the odd dimensional case. 	Let $\Delta$ be an s.e.d. $(2n-1)$-sphere. The generic anisotropy of $\Delta$ follows from Proposition \ref{prop:sed} by induction on both the dimension and the vertex number of $\Delta$. Note that if $\Delta$ is the boundary of a simplex, then $\Delta$ is generically anisotropic since $\kk(\Delta)=\kk[x]/(x^{2n+1})$ for any field $\kk$, and that $\mathrm{lk}_\Delta\sigma$ is an s.e.d. $(2n-3)$-sphere. Also, recall that simplicial $1$-spheres are generically anisotropic by \cite{PP23}, and that s.e.d. spheres are strong Lefschetz over any infinite field by \cite{Murai10}.

The even dimensional case can be deduced from Propositions \ref{prop:sus} and  \ref{prop:sed sus}. 

\begin{proof}[Proof of Proposition \ref{prop:sed sus}]
	We use induction on $d$. If $d=1$, the statement clearly holds. For the induction step, note that there are the following two easily verified facts:
	\begin{enumerate}[(a)]
		\item\label{fact:1} $\CC(\sigma,S(\Delta))=S[\CC(\sigma,\Delta)]$.
		\item\label{fact:2} If $\Delta$ satisfies the link condition with respect to an edge $\sigma\in \Delta$, then $S(\Delta)$ also satisfies the link condition with respect to $\sigma$.
		\end{enumerate}  
	 So if there is a sequence of simplicial $(d-1)$-spheres: \[\Delta=\Delta_0,\,\Delta_1,\dots,\Delta_s=\partial\Gamma,\ \Gamma\text{ a $d$-simplex},\]
	 where $\Delta_{i+1}=\CC(\sigma_i,\Delta_{i})$ for some edge $\sigma_i\in\Delta_i$  such that $\Delta_i$ satisfies the link condition with respect to $\sigma_i$ and $\mathrm{lk}_{\Delta_i}\sigma_i$ is an s.e.d. sphere,
   then $S(\Delta_0),\dots,S(\Delta_s)$ is a sequence of simplicial $d$-spheres satisfying the same conditions by the above facts \eqref{fact:1}, \eqref{fact:2} and the induction hypothesis for $\mathrm{lk}_{S(\Delta_i)}\sigma_i=S(\mathrm{lk}_{\Delta_i}\sigma_i)$, the suspension of an s.e.d. $(d-3)$-sphere. Then the induction step is completed by using the fact that  $\CC(\sigma,S(\partial\Gamma))$ is the boundary complex of a $(d+1)$-simplex  for any edge $\sigma\in S(\partial\Gamma)\setminus\partial\Gamma$.
\end{proof}

The proof of Proposition \ref{prop:sed} needs the following 
\begin{lem}\label{lem:basis}
	Let $\Delta$ be a $\kk$-homology $(2n-1)$-sphere, and suppose that $\sigma=\{u,v\}$ is an edge of $\Delta$ such that $\mathrm{lk}_\Delta\sigma$ is strong Lefschetz  over $\kk$. Let $\Theta=\{\theta_1,\dots,\theta_{2n}\}$ be a generic l.s.o.p. for $\kk[\Delta]$ such that the submatrix $(\boldsymbol{\lambda}_u,\boldsymbol{\lambda}_v)$ of $M_\Theta$ has the form $\begin{pmatrix}
		T\\
		0
	\end{pmatrix}$, where $T\in GL(\kk,2)$ is upper triangular, and let $\Theta_0=\{\theta_3,\dots,\theta_{2n}\}$. 
	Assume that $\{\rho_1,\dots,\rho_r\}$, $\rho_i\in\mathrm{lk}_\Delta\sigma$, is a basis for $\kk(\mathrm{lk}_\Delta\sigma;\Theta_0)_{n-1}$.  Then $\kk(\Delta;\Theta)_n$ has a basis of the form:
	\[\Aa\cup\BB,\ \Aa=\{\sigma_1,\dots,\sigma_r\},\ \BB=\{\tau_1,\dots,\tau_s\},\]
	where $\sigma_i=\{u\}\cup\rho_i$ for $1\leq i\leq r$ and $\tau_i\in\Delta\setminus\mathrm{st}_\Delta\sigma$, $v\not\in\tau_i$ for $1\leq i\leq s$. 
\end{lem}
\begin{proof}
 By Corollary \ref{cor:exact} there is a short exact sequence
	\[0\to \kk(\Delta,\mathrm{st}_\Delta\{v\})\to\kk(\Delta)\to\kk(\mathrm{st}_\Delta\{v\})\to 0,\]
	which shows that $\kk(\Delta)_n$ has a basis of the form: 
	$\Aa'\cup\BB'$, where $\Aa'\subset \mathrm{st}_\Delta\{v\}$ and $\BB'\subset \Delta\setminus\mathrm{st}_\Delta\{v\}$. We will show that $\Aa'$ can be chosen such that $\Aa'=\Aa\cup \BB''$, where elements in $\BB''$ satisfy the same condition for the ones in $\BB$. The Lemma will follow by taking $\BB=\BB'\cup\BB''$.
	
Assume that $T=\begin{pmatrix}
	a&b\\
	0&c
\end{pmatrix}$. Let $\omega=\theta_1-\frac{b}{c}\cdot\theta_2$, and let $\Theta'=\{\omega,\theta_3,\dots,\theta_{2n}\}$. Then $\Theta'$ is an l.s.o.p. for $\kk[\mathrm{lk}_\Delta\{v\}]$ since $\Theta$ is generic. As we have seen in the proof of Propositon \ref{prop:sus}, the natural inclusion  $\kk(\mathrm{lk}_\Delta\{v\};\Theta')\to\kk(\mathrm{st}_\Delta\{v\};\Theta)$ is an isomorphism. 
	So a basis for $\kk(\mathrm{lk}_\Delta\{v\};\Theta')_n$ is also a basis for $\kk(\mathrm{st}_\Delta\{v\};\Theta)_n$.
	
	Let $D$ be the closure of $\mathrm{lk}_\Delta\{v\}\setminus \mathrm{st}_\Delta\sigma$. Then $D$ is a $\kk$-homology $(2n-2)$-ball with $\partial D=\mathrm{lk}_\Delta\sigma$. Applying Corollary \ref{cor:exact} again we get a short exact sequence:
	\[0\to \kk(\mathrm{lk}_\Delta\{v\},D)\to\kk(\mathrm{lk}_\Delta\{v\})\to\kk(D)\to 0.\]
	Note that $\kk(\mathrm{lk}_\Delta\{v\},D)\cong \kk(\{u\}*\mathrm{lk}_\Delta\sigma,\mathrm{lk}_\Delta\sigma)$, and recall that there is an isomorphisms:
	 \[\cdot x_u:\kk(\mathrm{lk}_\Delta\sigma;\Theta_0)\to\kk(\{u\}*\mathrm{lk}_\Delta\sigma,\mathrm{lk}_\Delta\sigma;\Theta').\]
	 Hence $\Aa$ is a basis for $\kk(\mathrm{lk}_\Delta\{v\},D;\Theta')_n$. 
	 
	 To get the desired basis $\Aa'$ for $\kk(\mathrm{lk}_\Delta\{v\};\Theta')_n$, we will give a basis $\BB''\subset D\setminus\partial D$ for $\kk(D;\Theta')_n$. Consider another exact sequence:
	 \[\kk(D,\partial D;\Theta')\to\kk(D;\Theta')\to\kk(\partial D;\Theta_0)/(\omega)\to 0.\]
	 Since $\partial D=\mathrm{lk}_\Delta\sigma$ is strong Lefschetz over $\kk$ and $\Theta$ is generic, we may assume that $\omega$ is a strong Lefschetz element for $\kk(\partial D;\Theta_0)$. It follows that $(\kk(\partial D;\Theta_0)/(\omega))_n=0$, and so $\kk(D,\partial D;\Theta')_n\to\kk(D;\Theta')_n$ is a surjection. Thus a basis $\BB''$ of $\kk(D;\Theta')_n$ can be taken from $D\setminus\partial D$.
	 We conclude that $\Aa'=\Aa\cup \BB''$ is the desired basis for $\kk(\mathrm{st}_\Delta\{v\};\Theta)_n$, and the proof is finished.
 \end{proof}
 
 Before going into the proof of Proposition \ref{prop:sed}, we state an easy result about rational functions without proof.
 Let $\kk$ be a field, and let $\kk(t)$ be the  field of rational functions over $\kk$ with one variable $t$. For a nonzero element $\phi=f/g\in\kk(t)$ with $f,g\in\kk[t]$,
 define the \emph{degree} of $\phi$ by $\deg(\phi)=\deg(f)-\deg(g)$, and define the \emph{leading coefficient} of $\phi$ as $L(\phi)=L(f)/L(g)$, where $\deg(h)$ and $L(h)$ are the degree and leading coefficient of the polynomial $h\in\kk[t]$ respectively. Moreover, we assume $\deg(0)=-\infty$ and $L(0)=0$ in $\kk(t)$.
 \begin{lem}\label{lem:leading}
 	Let $\kk(t)$ be as above. Then for a nonzero element $\alpha=\sum_{i\in I}\phi_i$ with $\phi_i\in\kk(x)$, we have 
 	\[\deg(\alpha)\leq M:=\max\{\deg(\phi_i):i\in I\},\]
 	where equality holds if and only if $\sum_{\deg(\phi_i)=M}L(\phi_i)\neq0$.
 \end{lem}

\begin{proof}[Proof of Proposition \ref{prop:sed}]
	Assume $\sigma=\{1,2\}$. Let $\kk_0$ be the field of rational functions 
	\[\mathbb{F}(a_{i,j}:1\leq i\leq 2n,\,3\leq j\leq m),\] and let $\kk=\kk_0(t)$ be the field of rational funcitons over $\kk_0$ with one variable $t$. 
	For an element $f\in\kk$, we use $\deg(f)$ and $L(f)$ to denote the degree and the leading coefficient of $f$ with respect to $t$ respectively (see the definitions before Lemma \ref{lem:leading}). 
	
	Let $\Theta$ be the l.s.o.p. for $\kk[\Delta]$ such that in the matrix $M_\Theta=\{\boldsymbol{\lambda}_1,\dots,\boldsymbol{\lambda}_m\}$, $\boldsymbol{\lambda}_1=(1,0,\dots,0)^T$, $\boldsymbol{\lambda}_2=(t,1,0,\dots 0)^T$ and $\boldsymbol{\lambda}_j=(a_{1,j},a_{2,j},\dots,a_{2n,j})^T$ for $3\leq j\leq m$. 
	By the proof of Theorem \ref{thm:aniso equiv}, we only need to show that $\kk(\Delta;\Theta)$ is anisotropy. Acturally, it suffices to prove that the quadratic form $\kk(\Delta;\Theta)_n\times \kk(\Delta;\Theta)_n\to \kk(\Delta;\Theta)_{2n}\cong\kk$ is anisotropic. 
	To see this, note that if $0\neq\alpha\in\kk(\Delta)_i$ for $i<n$, then there exits $\alpha'\in\kk(\Delta)_{n-i}$ such that $0\neq\alpha\alpha'\in\kk(\Delta)_n$ sicne $\kk(\Delta)$ is a Poincar\'e duality $\kk$-algebra generated by degree one elements. 
	
	Let $\Delta'=\CC(\sigma,\Delta)$ with vertex set $[m]\setminus\{2\}$, and let $M'$ be the matrix obtained from $M_\Theta$ by deleting the column $\boldsymbol{\lambda}_2$. Then the set $\Theta'$ of one forms associated to the row vectors of $M'$ is an l.s.o.p. for $\kk_0[\Delta']$. 
	Let $\Gamma=\mathrm{lk}_\Delta\sigma$, and let $\Theta_0$ be the l.s.o.p. for $\kk_0[\Gamma]$ such that the matrix $M_{\Theta_0}$ is obtained from $M_\Theta$ by deleting the first two rows and restricting to the vertices of $\Gamma$.
		
	Assume $\Delta'$ and $\mathrm{lk}_\Delta\sigma$ are generically anisotropic over $\mathbb{F}$. Then $\kk_0(\Delta';\Theta')$ and $\kk_0(\Gamma;\Theta_0)$ are anisotropic by the proof of Theorem \ref{thm:aniso equiv}. Assume further that $\Gamma$ is strong Lefschetz over $\kk$. Then by Lemma \ref{lem:basis}, $\kk(\Delta;\Theta)_n$ has a basis of the form: 
	\[\Aa\cup\BB,\ \Aa=\{\sigma_1,\dots,\sigma_r\},\ \BB=\{\tau_1,\dots,\tau_s\},\]
	where $\sigma_i=\{1\}\cup\rho_i$, $\rho_i\in \mathrm{lk}_\Delta\sigma$ for $1\leq i\leq r$ and $\tau_j\in\Delta\setminus\mathrm{st}_\Delta\sigma$, $2\not\in\tau_j$ for $1\leq j\leq s$.

	We have the following facts:
	\begin{enumerate}[(i)]
		\item\label{item:1} $\Psi_\Delta(\xx_{\sigma_i}\xx_{\sigma_j})=\pm\Psi_{\Gamma}(\xx_{\rho_i}\xx_{\rho_j})t+b$, with $b\in\kk_0$, for  $1\leq i,j\leq r$.  Here $\Psi_{\Gamma}$ is defined on $\kk_0(\Gamma;\Theta_0)_{2n-2}$. 
		
		\item\label{item:2} $\Psi_\Delta(\xx_{\sigma_i}\xx_{\tau_j})=\Psi_{\Delta'}(\xx_{\sigma_i}\xx_{\tau_j})$ for $1\leq i\leq r,\, 1\leq j\leq s$. Here $\Psi_{\Delta'}$ is defined on $\kk_0(\Delta';\Theta')_{2n}$.
		
		\item\label{item:3} 
		$\deg(\Psi_\Delta(\xx_{\tau_i}\xx_{\tau_j}))\leq 0$ for $1\leq i,j\leq s$, and the equality holds if and only if $\Psi_{\Delta'}(\xx_{\tau_i}\xx_{\tau_j})\neq 0$, in which case  $L(\Psi_\Delta(\xx_{\tau_i}\xx_{\tau_j}))=\Psi_{\Delta'}(\xx_{\tau_i}\xx_{\tau_j})$.
	\end{enumerate}

Assume these facts for the moment. For any element $\alpha\in \kk(\Delta;\Theta)_n$, write $\alpha=\alpha_1+\alpha_2$, where 
\[\alpha_1=\sum_{i=1}^rl_i\xx_{\sigma_i},\ l_i\in\kk,\ \text{ and }\ \alpha_2=\sum_{j=1}^sk_j\xx_{\tau_j},\ k_j\in\kk.\]
Now define
 \begin{gather*}
	m_1=\max\{\deg(l_i):1\leq i\leq r\},\ \ m_2=\max\{\deg(k_j):1\leq j\leq s\},\\ 
	\beta_1=\sum_{\deg(l_i)=m_1}L(l_i)\xx_{\sigma_i},\  \beta_2=\sum_{\deg(k_j)=m_2}L(k_j)\xx_{\tau_j}.
\end{gather*}
If $m_1\geq m_2$, then $\beta_1\neq0$. Since $\kk_0(\Gamma;\Theta_0)$ is anisotropic, we have $\Psi_\Delta(\beta_1^2)=ft+b$, for some $f,b\in\kk_0$ with $f\neq 0$, by fact (i). Thus facts (i)-(iii) together with Lemma \ref{lem:leading} imply that  
\[\deg(\Psi_\Delta(\alpha^2))=\deg(\Psi_\Delta(\alpha_1^2))=2m_1+1\neq -\infty.\] On the other hand, if $m_1<m_2$, then $\beta_2\neq 0$. Hence by fact (iii) and Lemma \ref{lem:leading}, $\deg(\Psi_\Delta(\beta_2^2))=0$ and $L(\Psi_\Delta(\beta_2^2))=\Psi_{\Delta'}(\beta_2^2)\neq 0$ because of the anisotropy of $\kk_0(\Delta';\Theta')$. Using facts (i)-(iii) and Lemma \ref{lem:leading} again, we get 
\[\deg(\Psi_\Delta(\alpha^2))=\deg(\Psi_\Delta(\alpha_2^2))=2m_2\neq -\infty.\]
In either case we have $\alpha^2\neq 0$, then the proposition follows.

Now we prove \eqref{item:1}-\eqref{item:3}. For \eqref{item:1}, notice that for a facet $F\in\mathrm{st}_\Delta\{1\}$, if $2\not\in F$ then $A_F(l)$ ($l\in F$) and $A_F\in\kk_0$ (we may take $\mathbf{a}=\{1,1,\dots,1\}^T$ in the definition of $A_F(l)$), and if $2\in F$, then for $G=F\setminus\{1,2\}\in\Gamma$, $U=\sigma_i\cup\sigma_j$,  $V=\rho_i\cup\rho_j$, 
\[\frac{\prod_{l\in\sigma_i\cap\sigma_j}A_F(l)}{A_F\prod_{l\in F\setminus U}A_F(l)}=\pm\frac{A_F(1)}{A_F(2)}\cdot\frac{\prod_{l\in\rho_i\cap\rho_j}A_G(l)}{A_G\prod_{l\in G\setminus V}A_G(l)},\]
where $A_G$, $A_G(l)$ are defined on $M_{\Theta_0}$. A straightforward computation shows that $\frac{A_F(1)}{A_F(2)}=-t+c$ for some $c\in\kk_0$, and then \eqref{item:1} follows from Theorem \ref{thm:Lee}. 

\eqref{item:2} is obvious by the construction of $\Aa$ and $\BB$. 

For \eqref{item:3}, let $W=\tau_i\cup\tau_j$. If $W\not\in\mathrm{st}_\Delta\{2\}$, then $\Psi_\Delta(\xx_{\tau_i}\xx_{\tau_j})=\Psi_{\Delta'}(\xx_{\tau_i}\xx_{\tau_j})$ and the statement follows. So we assume $W\in\mathrm{st}_\Delta\{2\}$. If $F\in \mathrm{st}_\Delta\{2\}$ is a facet containing $W$, then it corresponds to a facet $F'\in \mathrm{st}_{\Delta'}\{1\}$ also containing $W$. It is easy to see that $A_F(2)=A_{F'}(1)$, $A_F=tA_{F'}+f$, $A_F(l)=tA_{F'}(l)+f_l$ for $2\neq l\in F$, where $f$, $f_l\in \kk_0$. 
Let
\[\phi_F=\frac{\prod_{l\in\tau_i\cap\tau_j}A_F(l)}{A_F\prod_{l\in F\setminus W}A_F(l)}\in\kk.\]
Then $\deg(\phi_F)=0$ since $|F\setminus W|=|\tau_i\cap\tau_j|$ and $2\in F\setminus W$. Moreover, $L(\phi_F)=\phi_{F'}$, where $\phi_{F'}\in\kk_0$ is defined in the same way as $\phi_F$. Hence \eqref{item:3} follows from Theorem \ref{thm:Lee} and Lemma \ref{lem:leading}.
\end{proof}

	\bibliography{M-A}
	\bibliographystyle{amsplain}
\end{document}